\documentclass[12pt,a4paper,reqno]{amsart}
\allowdisplaybreaks
\usepackage{amsmath}
\usepackage{amsfonts}
\usepackage{amssymb,amsthm,amsfonts,amsthm,latexsym,enumerate,url,cases}
\numberwithin{equation}{section}
\usepackage{mathrsfs}
\usepackage{hyperref}
\hypersetup{colorlinks=true,citecolor=blue,linkcolor=blue,urlcolor=blue}
\theoremstyle{plain}
\newtheorem{theorem}{Theorem}
\newtheorem{theorem'}{Theorem 4'}

\newtheorem{lemma}{Lemma}

\newtheorem{corollary}{Corollary}
\newtheorem{proposition}{Proposition}
\newtheorem{conjecture}{Conjecture}
\newtheorem{definition}{Definition}
\theoremstyle{definition}

\usepackage{etoolbox}
\makeatletter
\patchcmd{\@settitle}{\uppercasenonmath\@title}{}{}{}
\patchcmd{\@setauthors}{\MakeUppercase}{}{}{}
\patchcmd{\section}{\scshape}{}{}{}
\makeatother

\begin{document}

\title
[{Note on a theorem of Birch--Erd\H os and $m$-ary partitions}]
{Note on a theorem of Birch--Erd\H os and $m$-ary partitions\\
\bigskip
{\it \quad \quad\quad\quad \quad\quad\quad Dedicated to Prof. Yong-Gao Chen\\ \quad \quad\quad\quad \quad\quad\quad\quad \quad\quad \quad\quad\quad \quad\quad\quad\quad\quad \quad\quad\quad on the occasion of his retirement}}

\author
[Y. Ding, \quad  H. Liu \quad {\it and} \quad Z. Wang] {Yuchen Ding,\quad  Honghu Liu \quad {\it and} \quad Zi Wang}

\address{(Yuchen Ding) School of Mathematical Science,  Yangzhou University, Yangzhou 225002, People's Republic of China}
\email{ycding@yzu.edu.cn}

\address{(Honghu Liu) School of Mathematical Science,  Yangzhou University, Yangzhou 225002, People's Republic of China}
\email{3248952439@qq.com}

\address{(Zi Wang) Department of Mathematical Science,  HongKong University of science and technology, HongKong}
\email{zwanghk@connect.ust.hk}

\keywords{complete sequences, Birch--Erd\H{o}s theorem, $m$-ary partitions}
\subjclass[2010]{11A41}

\begin{abstract}
Let $p,q>1$ be two relatively prime integers and $\mathbb{N}$ the set of nonnegative integers. Let $f_{p,q}(n)$ be the number of different expressions of $n$ written as a sum of distinct terms taken from $\{p^{\alpha}q^{\beta}:\alpha,\beta\in \mathbb{N}\}$. Erd\H os conjectured and then Birch proved that $f_{p,q}(n)\ge 1$ provided that $n$ is sufficiently large. In this note, for all sufficiently large number $n$ we prove
$$
f_{p,q}(n)=2^{\frac{(\log n)^2}{2\log p\log q}\big(1+O(\log\log n/\log n)\big)}.
$$
We also show that $\lim_{n\rightarrow\infty}f_{2,q}(n+1)/f_{2,q}(n)=1.$
Additionally, 
we will point out the relations between $f_{2,q}(n)$ and $m$-ary partitions.
\end{abstract}
\maketitle

\section{Introduction}\label{section1}
Let $q>p>1$ be two relatively prime integers and $\mathbb{N}$ the set of nonnegative integers. Let also $f_{p,q}(n)$ be the number of different expressions of $n$ written as a sum of distinct terms taken from $\{p^{\alpha}q^{\beta}:\alpha,\beta\in \mathbb{N}\}$.
Erd\H os conjectured and then Birch \cite{Birch} proved $f_{p,q}(n)\ge 1$ providing that $n\ge n_{p,q}$, where $n_{p,q}$ is a given integer depending only on $p$ and $q$. Davenport observed from Birch's argument that $\beta$ could be bounded in terms of $p$ and $q$. Hegyv\'ari \cite{Hegyvari} give an explicit form of the bound of $\beta$ which was later improved by Bergelson and Simmons \cite{Bergelson}, Fang and Chen \cite{Fang}. Fang and Chen \cite{Fang} also obtained an upper bound of $n_{p,q}$.
For remarks on the Birch--Erd\H os theorem, one can refer to Erd\H os' article \cite{Erdos} (see also Bloom's problem list \# 246 \cite{Bloom} of Erd\H os). 

In this note, we give some further considerations of the Birch--Erd\H{o}s theorem.
Our first theorem is the following $\log$-asymptotic formula of $f_{p,q}(n)$.
\begin{theorem}\label{thm1}
Let $f_{p,q}(n)$ be defined as above. Then we have
$$
f_{p,q}(n)=2^{\frac{(\log n)^2}{2\log p\log q}\big(1+O(\log\log n/\log n)\big)}.
$$
\end{theorem}

In a former draft, we proved the following slightly weak one
$$
f_{p,q}(n)=2^{\left(\frac{1}{2}+o(1)\right)\frac{(\log n)^2}{\log p\log q}}.
$$
The present more explicit form of Theorem \ref{thm1} was obtained following the suggestion of Prof. Yong-Gao Chen.
From Theorem \ref{thm1} we have the following three corollaries.

\begin{corollary}\label{cor1}
Let $q>p>1$ be two relatively prime integers. Then 
$$
\log f_{p,q}(n)\sim \frac{\log 2(\log n)^2}{2\log p\log q}, \quad \text{as~}n\rightarrow\infty.
$$
\end{corollary}

\begin{corollary}\label{cor2}
Let $q_2>q_1>p>1$ be integers with $\gcd(p,q_1)=\gcd(p,q_2)=1$. Then 
$$
\lim_{n\rightarrow\infty}f_{p,q_2}(n)/f_{p,q_1}(n)=0.
$$
In particular, we have
$$
f_{p,q_1}(n)>f_{p,q_2}(n)
$$
for all sufficiently large $n$. 
\end{corollary}

\begin{corollary}\label{cor3}
Let $q>p>1$ be two relatively prime integers. Then 
$$
f_{p,q}(n+1)>f_{p,q}(n)
$$
for infinitely many $n$.
\end{corollary}

Neglecting error terms of the approximate formula of $f_{p,q}(n)$, we consider the function 
$$
\widetilde{f_{p,q}}(n)=2^{\frac{(\log n)^2}{2\log p\log q}}.
$$ 
Let $\ell\ge 2$ be a given positive integer. One notes easily that
\begin{align*}
\widetilde{f_{p,q}}(\ell n)=2^{\frac{(\log \ell n)^2}{2\log p\log q}}\ll_{p,q,\ell}2^{\frac{(\log n)^2}{2\log p\log q}}2^{\frac{\log \ell\log n}{\log p\log q}}=n^{\frac{\log 2\log \ell}{\log p\log q}}\widetilde{f_{p,q}}(n),
\end{align*}
from which we make the following bold conjecture.

\begin{conjecture}\label{conjecture1}
Let $q>p>1$ be two relatively prime integers and $\ell\ge 2$ a given integer. Then for any $\varepsilon>0$ we have
$$
f_{p,q}(\ell n)\ll n^{\frac{\log 2\log \ell}{\log p\log q}+\varepsilon}f_{p,q}(n),
$$
where the implied constant depends only on $p,q,\ell,$ and $\varepsilon$.
\end{conjecture}

Conjecture \ref{conjecture1} can be proved if the $O(\log\log n/\log n)$ term in Theorem \ref{thm1}  is improved to $o(1/\log n)$.

Corollary \ref{cor3} leads us to a natural problem. Is it true that 
$$
f_{p,q}(n+1)<f_{p,q}(n)
$$
for infinitely many $n$? It seems that the answers to this problem are quite different according to whether $p=2$ or not.
For $p=2$, we have the following results.

\begin{theorem}\label{thm2}
Let $q>2$ be an odd integer. Then for any $n\ge 1$, we have 
\begin{align*}
f_{2,q}(n+1)=
\begin{cases}
f_{2,q}(n), &~\text{if~}q\nmid n+1;\\
f_{2,q}(n)+f_{2,q}\big(\frac{n+1}{q}\big), &~\text{otherwise.}
\end{cases}
\end{align*}
\end{theorem}

The following Corollaries \ref{cor4}, \ref{cor5}, \ref{cor6}, and \ref{cor7} can be deduced from Theorem \ref{thm2} immediately.
\begin{corollary}\label{cor4}
Let $q>2$ be an odd integer. Then for any $n\ge 1$, we have  
$$
f_{2,q}(n+1)\ge f_{2,q}(n).
$$
\end{corollary}

\begin{corollary}\label{cor5}
Let $q>2$ be an odd integer. Then for any $n\ge 1$, we have  
$$
\lim_{x\rightarrow\infty}\frac{\#\{n\le x:f_{2,q}(n+1)= f_{2,q}(n)\} }{x}=\frac{q-1}{q}.
$$
\end{corollary}

\begin{corollary}\label{cor6}
Let $q>2$ be an odd integer. Suppose that
$
n=qm+r
$
with $0\le r<q$ and $m\ge 1$,
then we have
$$
f_{2,q}(n)=1+f_{2,q}(1)+f_{2,q}(2)+\cdot\cdot\cdot+f_{2,q}(m).
$$
\end{corollary}

\begin{corollary}\label{cor7}
Let $q_2>q_1$ be odd integers. Then for any $n\ge 1$ we have
$$
f_{2,q_1}(n)\ge f_{2,q_2}(n).
$$
\end{corollary}

We are also interested in the parity problems of $f_{p,q}(n)$. As another application of Theorem \ref{thm2}, our next theorem states as follows.

\begin{theorem}\label{thm3}
Let $q>2$ be odd integer. Then we have
$$
\liminf_{x\rightarrow\infty}\frac{\#\big\{n\le x:2|f_{2,q}(n)\big\} }{x}\ge \frac{1}{2q}-\frac{1}{2q^2}
$$
and
$$
\#\big\{n\le x:2\nmid f_{2,q}(n)\big\}\gg_q x^{\frac{\log \frac{q-1}{2}}{\log q}}
$$
for $q\ge 5$. Moreover, $f_{2,3}(n)$ take odd values infinitely many often.
\end{theorem}
Hence, both even and odd values will be taken by $f_{2,q}(n)$ infinitely many often. Following mathematical experiments, the following conjecture seems to be reasonable.

\begin{conjecture}\label{conjecture2}
We have
$$
\lim_{x\rightarrow\infty}\frac{\#\big\{n\le x:2|f_{2,3}(n)\big\} }{x}= \frac{1}{2}.
$$
\end{conjecture}

Unluckily, at present we are not able to establish parallel results for $p>2$. We make some related conjectures which seem to be supported by mathematical experiments.

\begin{conjecture}\label{conjecture4}
If $\gcd(p,q)=1$ and $q>p>2$, then the sign of
$$
f_{p,q}(n+1)-f_{p,q}(n)
$$
changes infinitely many often.
\end{conjecture}

Conjecture \ref{conjecture4}, if true, would provide another example of the famous Chebyshev bias phenomenon \cite{Chebyshev}. One can also compare Conjecture \ref{conjecture4} with Corollary \ref{cor4}.

\begin{conjecture}\label{conjecture3}
We have
$
f_{p,q}(n+1)=f_{p,q}(n)
$
for infinitely many $n$.
\end{conjecture}

Nearly the completeness of our manuscript, we used the recurrence formula obtained in Theorem \ref{thm2} to generate the sequence $\{f_{2,3}(n)\}_{n=1}^{\infty}$ and then searched for them from OEIS. We found out that the sequence $\{f_{2,3}(n)\}_{n=1}^{\infty}$ coincides with sequence A062051 of OEIS. A little thought persuades us that they are indeed the same sequences. We then further recognized the relations between $f_{2,q}(n)$ and $m$-ary partitions.

In A062051 of `{\it The On-Line Encyclopedia of Integer Sequence}' (OEIS) founded in 1964 by N. J. A. Sloane, it recorded the sequence $a(n)$ defined as the number of partitions of $n$ into powers of $3$. The sequence $a(n)$ could be further extended to the general ones known as the $m$-ary partition function (see, e.g. \cite{Andrews,Mahler,Rodseth}) defined below.

\begin{definition}\label{def1}
Let $m\ge 2$ be a given integer. Define the $m$-ary partition function $b_m(n)$ to be the number of partitions of $n$ into powers of $m$.
\end{definition}

We will show that the sequence $\big\{f_{2,q}(n)\big\}$ coincides with the sequence $\big\{b_{q}(n)\big\}$.

\begin{proposition}\label{pro1}
Let $q>2$ be an odd number. Then $f_{2,q}(n)=b_{q}(n)$ for any $n\ge1$.
\end{proposition}
\begin{proof}
Suppose that $n$ has a partition into powers of $q$ as
$
n=\sum_{i=0}^{k}a_iq^i
$
with $a_i\ge 1$. Clearly, $a_i$ has a unique $2$-adic expansion 
$
a_i=\sum_{j=0}^{g_i}\varepsilon_{i,j} 2^j
$
with $\varepsilon_j=0$ or $1$. Hence,
$$
n=\sum_{i=0}^{k}q^i\sum_{j=0}^{g_i}\varepsilon_{i,j} 2^j=\sum_{i,j}2^jq^i
$$
is an expression of $n$ counted by $f_{2,q}(n)$, which means $b_{q}(n)\le f_{2,q}(n)$.

Now, suppose that $n$ has an expression $n=\sum_{\alpha,\beta}2^\alpha q^\beta$ such that all terms in the expression are different. For the same $\beta$, we collect all the corresponding $2^\alpha$ and then sum them together to obtain a partition of $n$ into powers of $q$. This implies that $b_{q}(n)\ge  f_{2,q}(n)$. Therefore, we conclude that $f_{2,q}(n)=b_{q}(n)$.
\end{proof}

We noted from \cite[Page 2]{Rodseth} that 
$$
b_q(qn)-b_q(n)=b_q(qn-1)
$$
and 
$$
b_q(qn-1)=b_q(qn-2)=\cdots=b_q(qn-q),
$$
which together with Proposition \ref{pro1} would lead to Theorem \ref{thm2}. It seems (by combining with Proposition \ref{pro1}) that the recurrence formula displayed in Theorem \ref{thm2} was first established by Mahler \cite[Line -3, Page 8]{Mahler} as early as 1940.
Since we recognized the relations between $f_{2,q}(n)$ and $b_q(n)$ only after we nearly completed the manuscript as mentioned above, we will give an alternative proof of Theorem \ref{thm2} later in Section \ref{section3}. Theorem \ref{thm2} was proved by Mahler via generation functions \cite[Section 5, Page 8]{Mahler} while our proof of Theorem \ref{thm2} focuses on the combinatorial feature of $f_{2,q}(n)$.

It is also worth mentioning that Mahler \cite[the last formula of his article]{Mahler} (in our notations) obtained that
$$
\log f_{2,q}(n)\sim \frac{(\log n)^2}{2\log q}, \quad \text{as~}n\rightarrow\infty.
$$
Thus, Corollary \ref{cor1} can be viewed as an extension of Mahler's old result. Mahler's article is also related to the $m$-ary partitions. However, it does not seem easy to give a natural extension of Mahler's proof to the $\log$-asymptotic formula of $f_{p,q}(n)$. Another interesting point is that Mahler's proof of the $\log$-asymptotic formula of $f_{2,q}(n)$ is based on the theory of integral functions which is not elementary at all. While our extended $\log$-asymptotic formula of $f_{p,q}(n)$ is based on a recent result of Yu \cite{Yu} (see Lemma \ref{lemma} below) whose argument is entirely elementary.

Our last theorem surpasses Theorem \ref{thm1} in some aspects. 

\begin{theorem}\label{thm4}
Let $q>2$ be an odd integer. Then we have
$$
\lim_{n\rightarrow\infty}\frac{f_{2,q}(n+1)}{f_{2,q}(n)}=1.
$$
\end{theorem}

Since the proofs of Theorem \ref{thm4} only uses the recurrence formula of Theorem \ref{thm2},
it is perhaps of particular interest to point out that by the same arguments of Theorem \ref{thm4} the following slightly general result of $m$-ary partitions holds, which sees never to have been mentioned in literatures.

\begin{corollary}\label{cor8}
Let $m\ge2$ be a given integer. Then we have
$$
\lim_{n\rightarrow\infty}\frac{b_m(n+1)}{b_{m}(n)}=1.
$$
\end{corollary}

\section{Proof of Theorem \ref{thm1}}\label{section2}

The proof of Theorem \ref{thm1} is based on the following nice result of Yu \cite[Theorem 1.1]{Yu}.

\begin{lemma}\label{lemma} 
Let $p,q>1$ be integers with $\gcd(p,q)=1$. Then, for any real number $\epsilon>0$, there exists a positive real number $c=c(p,q,\epsilon)$ such that every sufficiently large integer $n$ can be represented as a sum of distinct integers of the form $p^{\alpha}q^{\beta}$ ($\alpha,\beta\in\mathbb{N}$), all of which are greater than
$$
cn/(\log n)^{1+\epsilon}.
$$
\end{lemma}

\begin{proof}[Proof of Theorem \ref{thm1}]
Throughout the proof, $n$ is supposed to be sufficiently large.
 
We first give the upper bound of $f_{p,q}(n)$. The number of pairs $(\alpha,\beta)\in \mathbb{N}^2$ such that $p^{\alpha}q^{\beta}\le n$ can be computed as 
\begin{align}\label{eq1}
\sum_{p^{\alpha}q^{\beta}\le n}1&=\sum_{\alpha \le \log n/\log p}\sum_{\beta\le \log (n/p^\alpha)/\log q}1\nonumber\\
&=\sum_{\alpha \le \log n/\log p}\left(\left\lfloor \frac{\log n-\alpha\log p}{\log q}\right\rfloor+1\right)\nonumber\\
&=\frac{(\log n)^2}{2\log p\log q}+O(\log n).
\end{align}
Let $\mathscr{E}=\left\{(\alpha,\beta)\in \mathbb{N}^2: p^{\alpha}q^{\beta}\le n\right\}$. Then by (\ref{eq1}) we have 
$$
|\mathscr{E}|= \frac{(\log n)^2}{2\log p\log q}+O(\log n),
$$
from which we obtain
$$
f_{p,q}(n)\le 2^{|\mathscr{E}|}=2^{\frac{(\log n)^2}{2\log p\log q}\big(1+O(1/\log n)\big)}.
$$

For the lower bound of $f_{p,q}(n)$, let 
$$
\widetilde{\mathscr{E}}=\left\{(\alpha,\beta)\in \mathbb{N}^2: p^{\alpha}q^{\beta}\le n/(\log n)^4\right\}.
$$
Then, by the same argument as (\ref{eq1}) (i.e., replacing $n$ by $n/(\log n)^4$) we have
\begin{align}\label{eq2}
\big|\widetilde{\mathscr{E}}\big|&= \frac{\Big(\log n-\log (\log n)^4\Big)^2}{2\log p\log q}+O(\log n)\nonumber\\
&=\frac{(\log n)^2}{2\log p\log q}+O(\log n\log\log n).
\end{align}
Let $P^{\widetilde{\mathscr{E}}}=\left\{S:S\subset \widetilde{\mathscr{E}}\right\}$. Then for any $S\in P^{\widetilde{\mathscr{E}}}$, by (\ref{eq2}) we clearly have
\begin{align}\label{eq3}
\sum_{(\alpha,\beta)\in S}p^\alpha q^\beta\le \frac{n}{(\log n)^4}|S|\le \frac{n}{(\log n)^4}\big|\widetilde{\mathscr{E}}\big|\ll \frac{n}{(\log n)^4}(\log n)^2=\frac{n}{(\log n)^2}.
\end{align}
From (\ref{eq3}) we know that 
\begin{align}\label{eq4}
\widetilde{n}:=n-\sum_{(\alpha,\beta)\in S}p^\alpha q^\beta\ge n-\frac{n}{\log n}>\frac12 n
\end{align}
for any $S\in \widetilde{\mathscr{E}}$. 

Now, we are in a position to make use of Lemma \ref{lemma}.
By Lemma \ref{lemma} it is clear that $\widetilde{n}$ has an expression of distinct integers of the form $p^{\alpha}q^{\beta}$ like
$$
\widetilde{n}=\sum_{(\alpha,\beta)\in \mathcal{D}}p^\alpha q^\beta,
$$
all of whose terms in the summation above are greater than
$
c\widetilde{n}/(\log \widetilde{n})^{1+\epsilon},
$
where $\mathcal{D}$ is a nonempty set, $\epsilon>0$ is arbitrarily small and $c$ is a constant depending only on $p,q$ and $\epsilon$. By (\ref{eq4}) (with $\epsilon=1/2$) we have
\begin{align}\label{eq5}
c\widetilde{n}/(\log \widetilde{n})^{1+\epsilon}\gg \frac{n}{(\log n)^{3/2}}.
\end{align}
Hence, for any $S\in P^{\widetilde{\mathscr{E}}}$ there corresponds with an expression of $n$ as
\begin{align}\label{eq6}
n=\sum_{(\alpha,\beta)\in S}p^\alpha q^\beta+\sum_{(\alpha,\beta)\in \mathcal{D}}p^\alpha q^\beta.
\end{align}
By (\ref{eq3}) and (\ref{eq5}), all terms $p^\alpha q^\beta$ in (\ref{eq6}) are different, which means that each $S\in P^{\widetilde{\mathscr{E}}}$ contributes a different expression of $n$.
So, there are at least $\big|P^{\widetilde{\mathscr{E}}}\big|$ different expressions of $n$ written as the desired forms. Therefore, we conclude that
$$
f_{p,q}(n)\ge \big|P^{\widetilde{\mathscr{E}}}\big|=2^{|\widetilde{\mathscr{E}}|}=2^{\frac{(\log n)^2}{2\log p\log q}\big(1+O(\log\log n/\log n)\big)}
$$
from (\ref{eq2}), proving our theorem.
\end{proof}

\section{A combinatorial proof of Theorem \ref{thm2}}\label{section3}
As mentioned in the abstract, the recurrence formula of Theorem \ref{thm2} first appeared in Mahler's article \cite{Mahler} (in view of Proposition \ref{pro1}). We recognized this fact only after we nearly completed our manuscript. Hence, we would like to provide our new proof of it.
 
From now on, let $q>2$ be a fixed odd integer. {\bf We denote $f_{2,q}(n)$ by $f(n)$ for brevity.} For any two disjoint subsets $\mathcal{A}$ and $\mathcal{B}$ of $\mathbb{N}$,
let $f_{\mathcal{A},\overline{\mathcal{B}}}(n)$ be the number of different expressions of $n$ written as
\begin{align}\label{eq3-1}
n=\sum_{\alpha,\beta}2^{\alpha}q^{\beta} \quad (2^{\alpha}q^{\beta} ~\text{are~all~different})
\end{align}
such that each element of $\mathcal{A}$ equals some term $2^{\alpha}q^{\beta}$ in (\ref{eq3-1}) while each element of $\mathcal{B}$ does not equal any term $2^{\alpha}q^{\beta}$ in (\ref{eq3-1}).
For example, $f_{\{q^3,2^5q\},\overline{\{2^2,q^2,2^5\}}}(n)$ will be the number of different expressions of $n$ written as $\sum_{\alpha,\beta}2^{\alpha}q^{\beta}$ such that $q^3$ and $2^5q$ are contained
in the expression while $2^2,q^2$ and $2^5$ do not occur in it.

For discussions, we will also need some similar notations. Define
$f_{\mathcal{A}}(n)$ to be the number of different expressions of $n$ written as (\ref{eq3-1})
such that each element of $\mathcal{A}$ equals some term $2^{\alpha}q^{\beta}$ in (\ref{eq3-1}). Define
$f_{\overline{\mathcal{B}}}(n)$ to be the number of different expressions of $n$ written as (\ref{eq3-1})
such that each element of $\mathcal{B}$ does not equal any term $2^{\alpha}q^{\beta}$ in (\ref{eq3-1}).

For the proof of Theorem \ref{thm2}, we need the following lemmas. 

\begin{lemma}\label{lemma2}
For any given integers $t\ge 0$ and $n\ge 1$ we have
$$
f_{\{2^{t+1}\},\overline{\{1,2,\cdots,2^{t}\}}}(n+1)=f_{\{1,2,\cdots,2^{t}\},\overline{\{2^{t+1}\}}}(n).
$$ 
\end{lemma}
\begin{proof}
If $2^{t+1}>n+1$, then $1+2+\cdots+\cdots+2^{t}=2^{t+1}-1>n$ and hence
$$
f_{\{2^{t+1}\},\overline{\{1,2,\cdots,2^{t}\}}}(n+1)=f_{\{1,2,\cdots,2^{t}\},\overline{\{2^{t+1}\}}}(n)=0.
$$
If $2^{t+1}=n+1$, then $1+2+\cdots+\cdots+2^{t}=n$ and hence
$$
f_{\{2^{t+1}\},\overline{\{1,2,\cdots,2^{t}\}}}(n+1)=f_{\{1,2,\cdots,2^{t}\},\overline{\{2^{t+1}\}}}(n)=1.
$$

We are left over to consider the case $2^{t+1}< n+1$. If $n+1$ has an expression
\begin{align}\label{eq3-2}
n+1=2^{t+1}+\sum_{\alpha,\beta}2^{\alpha}q^{\beta}, \quad (2^{\alpha}q^{\beta}\neq 1,2,\cdots,2^{t},2^{t+1})
\end{align}
then we have
\begin{align}\label{eq3-3}
n=1+2+\cdots+2^t+\sum_{\alpha,\beta}2^{\alpha}q^{\beta}, \quad (2^{\alpha}q^{\beta}\neq 1,2,\cdots,2^{t},2^{t+1})
\end{align}
and vice versa. By (\ref{eq3-2}) and (\ref{eq3-3}) we have
$$
f_{\{2^{t+1}\},\overline{\{1,2,\cdots,2^{t}\}}}(n+1)=f_{\{1,2,\cdots,2^{t}\},\overline{\{2^{t+1}\}}}(n),
$$ 
proving our lemma.
\end{proof}

\begin{lemma}\label{lemma3}
For any given integers $t\ge 0$ and $n\ge 1$ we have
$$
f(n+1)-f(n)=f_{\overline{\{1,2,2^2,\cdots,2^t\}}}(n+1)-f_{\{1,2,2^2,\cdots,2^t\}}(n).
$$ 
\end{lemma}
\begin{proof}
If $n+1$ has an expression
\begin{align*}
n+1=1+\sum_{\alpha,\beta}2^{\alpha}q^{\beta}, \quad (2^{\alpha}q^{\beta}\neq 1)
\end{align*}
then we have
\begin{align*}
n=\sum_{\alpha,\beta}2^{\alpha}q^{\beta}, \quad (2^{\alpha}q^{\beta}\neq 1)
\end{align*}
and vice versa, which means that
\begin{align}\label{eq3-4}
f_{\{1\}}(n+1)=f_{\overline{\{1\}}}(n).
\end{align}
Note that 
\begin{align}\label{eq3-5}
f(n+1)=f_{\{1\}}(n+1)+f_{\overline{\{1\}}}(n+1) \quad  \text{and}  \quad f(n)=f_{\{1\}}(n)+f_{\overline{\{1\}}}(n)
\end{align}
From (\ref{eq3-4}) and (\ref{eq3-5}), 
\begin{align}\label{eq3-6}
f(n+1)-f(n)=f_{\overline{\{1\}}}(n+1)-f_{\{1\}}(n).
\end{align}
Clearly, we have
\begin{align}\label{eq3-7}
f_{\overline{\{1\}}}(n+1)=f_{\{2\},\overline{\{1\}}}(n+1)+f_{\overline{\{1,2\}}}(n+1)
\end{align}
and
\begin{align}\label{eq3-8}
f_{\{1\}}(n)=f_{\{1,2\}}(n)+f_{\{1\},\overline{\{2\}}}(n).
\end{align}
By Lemma \ref{lemma2} (with $t=0$), (\ref{eq3-6}), we have
$$
f_{\{2\},\overline{\{1\}}}(n+1)=f_{\{1\},\overline{\{2\}}}(n),
$$
which together with (\ref{eq3-7}) and (\ref{eq3-8}) lead to
\begin{align}\label{eq3-9}
f(n+1)-f(n)=f_{\overline{\{1,2\}}}(n+1)-f_{\{1,2\}}(n).
\end{align}
Furthermore, noting that
\begin{align}\label{eq3-10}
f_{\overline{\{1,2\}}}(n+1)=f_{\{2^2\},\overline{\{1,2\}}}(n+1)+f_{\overline{\{1,2,2^2\}}}(n+1)
\end{align}
and
\begin{align}\label{eq3-11}
f_{\{1,2\}}(n)=f_{\{1,2,2^2\}}(n)+f_{\{1,2\},\overline{\{2^2\}}}(n).
\end{align}
By Lemma \ref{lemma2} (with $t=1$), we have
$$
f_{\{2^2\},\overline{\{1,2\}}}(n+1)=f_{\{1,2\},\overline{\{2^2\}}}(n),
$$
which together with (\ref{eq3-9}), (\ref{eq3-10}) and (\ref{eq3-11}) lead to
\begin{align*}
f(n+1)-f(n)=f_{\overline{\{1,2,2^2\}}}(n+1)-f_{\{1,2,2^2\}}(n).
\end{align*}
Continuing the process above, for any integer $t\ge 0$ we obtain
\begin{align}\label{eq3-12}
f(n+1)-f(n)=f_{\overline{\{1,2,2^2,\cdots,2^t\}}}(n+1)-f_{\{1,2,2^2,\cdots,2^t\}}(n),
\end{align}
which ends up with the argument of our lemma.
\end{proof}

We are ready to provide the proof of Theorem \ref{thm2}.

\begin{proof}[Proof of Theorem \ref{thm2}]
Let $n\ge 1$ be any given integer. Let $t$ be the minimal integer such that $2^{t+1}-1>n$. Then
$$
f_{\{1,2,2^2,\cdots,2^t\}}(n)=0.
$$
By Lemma \ref{lemma3}, we have
\begin{align}\label{eq3-13}
f(n+1)-f(n)=f_{\overline{\{1,2,2^2,\cdots,2^t\}}}(n+1).
\end{align}
From $2^{t+1}>n+1$ we know that there are no powers of $2$ occurred in any expression counted by $f_{\overline{\{1,2,2^2,\cdots,2^t\}}}(n+1)$. Suppose now that 
$$
f_{\overline{\{1,2,2^2,\cdots,2^t\}}}(n+1)\ge 1
$$ 
and $\sum_{\alpha,\beta}2^\alpha q^\beta$ is such an expression of $n+1$ counted by $f_{\overline{\{1,2,2^2,\cdots,2^t\}}}(n+1)$, then each term $2^\alpha q^\beta$ in the expression $\sum_{\alpha,\beta}2^\alpha q^\beta$ satisfies $\beta\ge 1$, which means that
$
q|n+1
$ and vice versa. Therefore, we conclude that
$$
f(n+1)-f(n)=f_{\overline{\{1,2,2^2,\cdots,2^t\}}}(n+1)=0
$$
if and only if $q\nmid n+1$. 

For $q| n+1$, we assume 
$
n+1=qk
$
and then $f_{\overline{\{1,2,2^2,\cdots,2^t\}}}(qk)\ge 1$. Since  $2^{t+1}>n+1=qk$, there are no powers of $2$ occurred in any expression counted by
$f_{\overline{\{1,2,2^2,\cdots,2^t\}}}(qk)$. Now, let $\sum_{\alpha,\beta}2^\alpha q^\beta$ be an expression of $qk$ counted by $f_{\overline{\{1,2,2^2,\cdots,2^t\}}}(qk)$, then each term $2^\alpha q^\beta$ in the expression $\sum_{\alpha,\beta}2^\alpha q^\beta$ satisfies $\beta\ge 1$. Therefore, from
$$
\sum_{\alpha,\beta}2^\alpha q^\beta=q\sum_{\alpha,\beta}2^\alpha q^{\beta-1}
$$
we deduce that 
\begin{align}\label{eq3-14}
f_{\overline{\{1,2,2^2,\cdots,2^t\}}}(qk)=f(k)=f\Big(\frac{n+1}{q}\Big).
\end{align}
Hence, we obtain
\begin{align*}
f(n+1)=
\begin{cases}
f(n), &~\text{if~}q\nmid n+1;\\
f(n)+f\big(\frac{n+1}{q}\big), &~\text{otherwise}
\end{cases}
\end{align*}
from (\ref{eq3-13}) and (\ref{eq3-14}).
\end{proof}

\section{Proof of Theorem \ref{thm3}}\label{section4}
Define $g(n)$ to be $g(0)=0$,
$$
g(1)=f(q-1)=f(q-2)=\cdots=f(1)=1
$$
and
$$
g(n)=f(qn-1)=f(qn-2)=\cdots=f(qn-q),
$$
for $n\ge2$.
Then we have 
\begin{align}\label{definition-g}
g(n+1)=g(n)+g\Big(\Big\lceil \frac{n+1}{q}\Big\rceil\Big)
\end{align}
from the recurrence formula of $f(n)$ in Theorem \ref{thm2},
where $\lceil x\rceil$ is the least integer $\ge x$.
We now provide the proof of Theorem \ref{thm3}.
\begin{proof}[Proof of Theorem \ref{thm3}]
By (\ref{definition-g}), for any nonnegative integer $m$ we have
\begin{align}\label{desity-0}
g(n+1)=g(n)+g(m+1)
\end{align}
provided that $mq<n+1\le (m+1)q$, or more clearly,
\begin{align}\label{desity-1}
g(mq+1)=g(mq)+g(m+1).
\end{align}

It can be easily seen that at least one of the three numbers $g(m+1), g(mq)$ and $g(mq+1)$ is even from (\ref{desity-1}). Note that for any positive integers $j\neq \ell$  with $q\nmid j$ and $q\nmid \ell$ we have 
$$
\Big\{2jq+1,2jq,2j+1\Big\} \cap \Big\{2\ell q+1,2\ell q,2\ell+1\Big\}=\varnothing.
$$
Hence, by $m=2i$ with $q\nmid i$ we conclude that
$$
\liminf_{x\rightarrow\infty}\frac{\#\big\{n\le x:2|g(n)\big\}}{x}\ge \liminf_{x\rightarrow\infty}\frac{\#\big\{i:q\nmid i, 2qi+2\le x\big\}}{x}=\frac{1}{2q}-\frac{1}{2q^2},
$$
from which we deduce that
$$
\liminf_{x\rightarrow\infty}\frac{\#\big\{n\le x:2|f(n)\big\}}{x}\ge \liminf_{x\rightarrow\infty}\frac{\#\big\{n\le x/q:2|g(n)\big\}\cdot q}{x}=\frac{1}{2q}-\frac{1}{2q^2}.
$$

Suppose that $g(m+1)$ is odd, then the next odd value of $g(n)$ will no later than $g(mq+1)$ by using (\ref{desity-1}) once again, from which it follows that $g(n)$ (and hence $f(n)$) take odd values infinitely many often. Moreover, from (\ref{desity-0}) at least $\frac{q-1}{2}$ of $g(n)$ take odd values for $mq\le  n< (m+1)q$, provided that $g(m+1)$ is odd. Actually, the parities of $g(n)$ alternate between odd and even for $mq\le n< (m+1)q$. Therefore, each odd value $g(n)$ produce at least $\frac{q-1}{2}$ new odd values $g(\widetilde{n})$ for $nq\le \widetilde{n}< (n+1)q$ and the $\widetilde{n}$'s are all different for different $n$. For odd $g(m+1)$, we further note that the worst parity situation (i.e., exactly $\frac{q-1}{2}$ odd values) of $q$ consecutive integers in the interval $[mq,(m+1)q)$ is
$$
\text{even,~odd,~even,~odd,}~\cdots,~\text{odd,~even}.
$$ 
For $q>3$ suppose that $q^k\le x<q^{k+1}$, then from the above analysis we know that
$$
\#\big\{n\le x:2\nmid g(n)\big\}\ge \sum_{i=1}^k\left(\frac{q-1}{2}\right)^i=\frac{q-1}{q-3}\bigg(\Big(\frac{q-1}{2}\Big)^k-1\bigg).
$$
Clearly,
\begin{align*}
\Big(\frac{q-1}{2}\Big)^k>\Big(\frac{q-1}{2}\Big)^{\frac{\log x}{\log q}-1}=\frac{2}{q-1}x^{\frac{\log \frac{q-1}{2}}{\log q}}.
\end{align*}
Hence, it follows that
\begin{align*}
\#\big\{n\le x:2\nmid g(n)\big\}\ge \frac{2}{q-3}x^{\frac{\log \frac{q-1}{2}}{\log q}}.
\end{align*}
Thus, we have
\begin{align*}
\#\big\{n\le x:2\nmid f(n)\big\}\ge \#\Big\{n\le \frac{x}{q}:2\nmid g(n)\Big\}\cdot q-1\gg x^{\frac{\log \frac{q-1}{2}}{\log q}},
\end{align*}
which is our desired bound.
\end{proof}

\section{Proof of Theorem \ref{thm4}}\label{section5}
For the proof, we make use of two auxiliary functions. Let $q\ge3$ be a fixed odd integer and $n\ge1$ be integers. The first auxiliary function is $g(n)$ defined in Section \ref{section4} and the other one is defined as follows.

Let $h(n)$ be defined as $h(1)=1, h(2)=2, \cdots, h(q)=q,$ and 
\begin{align*}
h(n+1)=
\begin{cases}
h(n)+1, &~\text{if~}q\nmid n;\\
h(n)\Big(1-\frac{1}{h\big(\frac{n}{q}+1\big)}\Big)+1, &~\text{otherwise.}
\end{cases}
\end{align*}
It can be easily seen that $h(n)\ge 2$ for any $n\ge 2$.
We need two lemmas.

\begin{lemma}\label{lemma4}
Let $\lfloor x\rfloor$ be the largest integer $\le x$. Then for any integer $n\ge 1$ we have
$$
h(n)\ge \Big\lfloor \frac{\log n}{\log q}\Big\rfloor.
$$ 
\end{lemma}
\begin{proof}
We prove it by induction on $n$. 
Clearly, $h(1)=1>0$ and we have
$$
h(n)\ge 2\ge \Big\lfloor \frac{\log n}{\log q}\Big\rfloor
$$
for $2\le n\le q^3-1$.
Now, assume that our lemma is true for every integer $\le n$. We need to prove the truth of $n+1$ with $n+1\ge q^3$.
The subsequent proofs will be separated into several cases.

{\it Case I.}  $n=q^k$ for some integer $k\ge 3$.  Using the inductive hypothesis, we know
$$
h\big(q^k-2\big)\ge \Big\lfloor \frac{\log (q^k-2)}{\log q}\Big\rfloor=k-1.
$$
It follows that
\begin{align*}
h(n)=h\big(q^k\big)=h\big(q^k-1\big)+1=h\big(q^k-2\big)+2\ge k+1.
\end{align*}
Hence, by inductive hypothesis and the definition of $h$ we obtain
\begin{align*}
h(n+1)&=h(n)\Big(1-\frac{1}{h\big(\frac{n}{q}+1\big)}\Big)+1\\
&\ge (k+1)\Big(1-\frac{1}{h\big(q^{k-1}+1\big)}\Big)+1\\
&\ge (k+1)\left(1-\frac{1}{\Big\lfloor \frac{\log (q^{k-1}+1)}{\log q}\Big\rfloor}\right)+1\\
&=(k+1)\Big(1-\frac{1}{k-1}\Big)+1\\
&\ge \Big\lfloor \frac{\log (n+1)}{\log q}\Big\rfloor,
\end{align*}
provided that $k\ge 3$.

{\it Case II.} $q\nmid n$. In this case, we have
\begin{align*}
h(n+1)=h(n)+1\ge \Big\lfloor \frac{\log n}{\log q}\Big\rfloor+1\ge \Big\lfloor \frac{\log (n+1)}{\log q}\Big\rfloor.
\end{align*}

{\it Case III.} $q|n$ and $n$ is not a power of $q$. We assume 
$
q^k<n<q^{k+1}
$
with $k\ge 3$.
Suppose that 
$$
q^\ell | n \quad \text{but} \quad q^{\ell+1}\nmid n
$$ 
for some $1\le \ell\le k$. We note from 
$
q^k<n<q^{k+1}
$
and $q|n$
that $n\le q^{k+1}-q$. 
For $\ell=1$, then $q|n$ and $q^2\nmid n$. 
By the definition of $h$ we have
\begin{align}\label{lemma5-4-2}
h(n+1)=h(n)\Big(1-\frac{1}{h\big(\frac{n}{q}+1\big)}\Big)+1.
\end{align}
Since $q\nmid \frac{n}{q}$, we know from the inductive hypothesis that 
\begin{align}\label{lemma5-4-3}
h\Big(\frac{n}{q}+1\Big)=h\Big(\frac{n}{q}\Big)+1\ge\Big\lfloor \frac{\log (n/q)}{\log q}\Big\rfloor +1=k-1+1=k.
\end{align}
Using again the inductive hypothesis, we conclude from (\ref{lemma5-4-2}) and (\ref{lemma5-4-3}) that
\begin{align*}
h(n+1)&\ge \Big\lfloor \frac{\log n}{\log q}\Big\rfloor\Big(1-\frac{1}{k}\Big)+1\\
&=k\Big(1-\frac{1}{k}\Big)+1\\
&=\Big\lfloor \frac{\log (n+1)}{\log q}\Big\rfloor
\end{align*}
since $n+1\le q^{k+1}-q+1<q^{k+1}$.

Now, suppose that $\ell\ge 2$, i.e., $q^2|n$. Hence, 
$$
q^k+q^2\le n\le q^{k+1}-q^2.
$$
By the definition of $h$ we have
\begin{align}\label{lemma5-4-4}
h(n+1)&=h(n)\Big(1-\frac{1}{h\big(\frac{n}{q}+1\big)}\Big)+1\nonumber\\
&=\big(h(n-1)+1\big)\Big(1-\frac{1}{h\big(\frac{n}{q}+1\big)}\Big)+1\nonumber\\
&=\big(h(n-2)+2\big)\Big(1-\frac{1}{h\big(\frac{n}{q}+1\big)}\Big)+1
\end{align}
Then by the inductive hypothesis we have
\begin{align}\label{lemma5-4-5}
h\Big(\frac{n}{q}+1\Big)\ge \Big\lfloor \frac{\log (n/q+1)}{\log q}\Big\rfloor=k-1.
\end{align}
Inserting (\ref{lemma5-4-5}) into (\ref{lemma5-4-4}), by $n+1<q^{k+1}$ and the inductive hypothesis we obtain
\begin{align*}
h(n+1)&\ge \Big(\Big\lfloor \frac{\log (n-2)}{\log q}\Big\rfloor+2\Big)\Big(1-\frac{1}{k-1}\Big)+1\\
&=(k+2)\Big(1-\frac{1}{k-1}\Big)+1\\
&=\frac{k^2+k-5}{k-1}\\
&\ge \Big\lfloor \frac{\log (n+1)}{\log q}\Big\rfloor,
\end{align*}
provided that $k\ge 3$.
\end{proof}

\begin{lemma}\label{lemma5}
Let $\lceil x\rceil$ be the least integer $\ge x$. Then for any integer $n\ge 1$ we have
$$
g(n)\ge h(n) g\Big(\Big\lceil \frac{n}{q}\Big\rceil\Big).
$$ 
\end{lemma}
\begin{proof}
We prove it by induction on $n$. 
We first show that our lemma is true for $1\le n\le q$.
For $1\le n\le q$, we clearly have 
$$
g\Big(\Big\lceil \frac{n}{q}\Big\rceil\Big)=g(1)=1.
$$
From the definition of $g$ and Theorem \ref{thm2}, we gave
$$
g(n+1)=f(qn+q-1)=f(qn)=f(qn-1)+f(n)\ge g(n)+1,
$$
from which we deduce that 
$$
g(n)\ge h(n)=h(n) g\Big(\Big\lceil \frac{n}{q}\Big\rceil\Big)
$$ 
since $h(n+1)=h(n)+1$ and $g(1)=h(1)=1$.

Now, suppose that our lemma is true for all the integers $\le n$, we are going to show the truth of $n+1$ with $n+1\ge q+1$. It can be separate into the following two cases.

{\it Case I.} $q\nmid n$. In this case, by (\ref{definition-g}) and the definition of $h$ we have
\begin{align*}
g(n+1)&=g(n)+g\Big(\Big\lceil \frac{n+1}{q}\Big\rceil\Big)\\
&\ge h(n) g\Big(\Big\lceil \frac{n}{q}\Big\rceil\Big)+g\Big(\Big\lceil \frac{n+1}{q}\Big\rceil\Big)\\
&=h(n) g\Big(\Big\lceil \frac{n+1}{q}\Big\rceil\Big)+g\Big(\Big\lceil \frac{n+1}{q}\Big\rceil\Big)\\
&=\big(h(n)+1\big)g\Big(\Big\lceil \frac{n+1}{q}\Big\rceil\Big)\\
&=h(n+1)g\Big(\Big\lceil \frac{n+1}{q}\Big\rceil\Big),
\end{align*}
proving the truth of $n+1$.

{\it Case II.} $q| n$. Suppose that $n=qm$, then by (\ref{definition-g}) and the definition of $h$  we have
\begin{align*}
g(n+1)&=g(n)+g\Big(\Big\lceil \frac{n+1}{q}\Big\rceil\Big)\\
&= g(qm) +g(m+1)\\
&\ge h(qm)g(m)+g(m+1)\\
&=h(qm)\bigg(g(m+1)-g\Big(\Big\lceil \frac{m+1}{q}\Big\rceil\Big)\bigg)+g(m+1)\\
&=\big( h(qm)+1\big)g(m+1)-h(qm)g\Big(\Big\lceil \frac{m+1}{q}\Big\rceil\Big).
\end{align*}
Note that
$$
g(m+1)\ge h(m+1)g\Big(\Big\lceil \frac{m+1}{q}\Big\rceil\Big)
$$
since $m<n$, from which it follows that
\begin{align*}
g(n+1)&\ge \big( h(qm)+1\big)g(m+1)-h(qm)\frac{g(m+1)}{h(m+1)}\\
&=\Big(h(qm)+1-\frac{h(qm)}{h(m+1)}\Big)g(m+1)\\
&=h(qm+1)g(m+1)\\
&=h(n+1)g\Big(\Big\lceil \frac{n+1}{q}\Big\rceil\Big),
\end{align*}
where the last but one equality follows from the definition of $h(n)$.
\end{proof}

We now provide the proof of Theorem \ref{thm4}.
\begin{proof}[Proof of Theorem \ref{thm4}]
For $q\nmid n+1$, we clearly have $f(n+1)=f(n)$ from Theorem \ref{thm2} and hence $f(n+1)/f(n)=1$. We are left over to consider $q|n+1$.

Suppose that $n+1=qk$, then by the definition of $g(n)$ we have
\begin{align}\label{th4-1}
\lim_{n\rightarrow\infty}\frac{f(n)}{f(n+1)}=\lim_{k\rightarrow\infty}\frac{f(qk-1)}{f(qk)}
=\lim_{k\rightarrow\infty}\frac{g(k)}{g(k+1)}.
\end{align}
Again, from the definition of $g$ we have 
\begin{align}\label{th4-2}
g(k)=g(k+1)-g\Big(\Big\lceil \frac{k+1}{q}\Big\rceil\Big).
\end{align}
Taking (\ref{th4-2}) into (\ref{th4-1}), we obtain
\begin{align}\label{th4-3}
\lim_{n\rightarrow\infty}\frac{f(n)}{f(n+1)}=\lim_{k\rightarrow\infty}\frac{g(k+1)-g\Big(\Big\lceil \frac{k+1}{q}\Big\rceil\Big)}{g(k+1)}=1-\lim_{k\rightarrow\infty}\frac{g\Big(\Big\lceil \frac{k+1}{q}\Big\rceil\Big)}{g(k+1)}.
\end{align}
By Lemmas \ref{lemma4} and \ref{lemma5} we know
\begin{align}\label{th4-4}
g(k+1)\ge h(k+1)g\Big(\Big\lceil \frac{k+1}{q}\Big\rceil\Big)\ge \left\lfloor \frac{\log (k+1)}{\log q}\right\rfloor g\Big(\Big\lceil \frac{k+1}{q}\Big\rceil\Big).
\end{align}
From (\ref{th4-3}) and (\ref{th4-4}) we conclude 
$$
\lim_{n\rightarrow\infty}\frac{f(n)}{f(n+1)}=1,
$$
which is our desired result.
\end{proof}

\section{Remarks}\label{section6}

In \cite{Erdos-Lewin}, Erd\H os and Lewin also considered some related problems which are now known as $d$-complete sequences. A subset $A$ of $\mathbb{N}$ is called $d$-complete if every sufficiently large integer is the sum of distinct terms taken from $A$ such that no one divides the other. For $q>p>1$ with $\gcd(p,q)=1$, Erd\H os and Lewin showed that $\{p^{\alpha}q^{\beta}:\alpha,\beta\in \mathbb{N}\}$ is $d$-complete if and only if $(p,q)=(2,3)$. Answering affirmatively a question of Erd\H os and Lewin, Yu and Chen \cite{Yu2} proved the following parallel result on Lemma \ref{lemma}: Every integer $n\ge 2$ can be represented as a sum of integers all of the form $2^\alpha 3^\beta$, all of which are greater than 
\begin{align}\label{eq-section}
\frac{n}{20(\log n)^{\log 3/\log 2}},
\end{align}
and none of which divides the other.  Very recently, the bound (\ref{eq-section}) is improved to $cn/\log n$ for some absolute constant $c$ by Yang and Zhao \cite{Yang}. For more related literatures on complete sequences, see e.g. \cite{CFH,Chen,Hegyvari3,Hegyvari2}.

Let $g(n)$ be the number of different expressions of $n$ written as a sum of distinct terms taken from $\{2^{\alpha}3^{\beta}:\alpha,\beta\in \mathbb{N}\}$ such that no one divides the other. 
Finding the $\log$-asymptotic formula of $g(n)$, if it exists, is however an unsolved problem. 

\section*{Acknowledgments}
We thank Yong-Gao Chen and Wang-Xing Yu for their interests of this article. 

The author is supported by National Natural Science Foundation of China  (Grant No. 12201544) and China Postdoctoral Science Foundation (Grant No. 2022M710121).

\end{document}